\documentclass[12pt, reqno]{amsart}

\usepackage{latexsym}
\usepackage{amssymb}
\usepackage{mathrsfs}
\usepackage{amsmath}
\usepackage{fancybox,color}
\usepackage{enumerate}
\usepackage[latin1]{inputenc}
\usepackage{eurosym}

\newcommand{\kom}[1]{}
\renewcommand{\kom}[1]{{\bf [#1]}}

 \def\1{\raisebox{2pt}{\rm{$\chi$}}}

\newtheorem{theorem}{Theorem}[section]
\newtheorem{corollary}[theorem]{Corollary}
\newtheorem{lemma}[theorem]{Lemma}
\newtheorem{proposition}[theorem]{Proposition}
\newtheorem{definition}[theorem]{Definition}

\newcommand{\N}{{\mathbb N}}
\newcommand{\Z}{{\mathbb Z}}

\newcommand{\E}{{\mathbb E\,}}
\newcommand{\ha}{{\mathcal{H}}}

 \newcommand{\eps}{{\varepsilon}}
 \def\1{\raisebox{2pt}{\rm{$\chi$}}}
 
\newcommand{\norm}[1]{\left|\left|#1\right|\right|}

\def\vint_#1{\mathchoice%
          {\mathop{\kern 0.2em\vrule width 0.6em height 0.69678ex depth -0.58065ex
                  \kern -0.8em \intop}\nolimits_{\kern -0.4em#1}}%
          {\mathop{\kern 0.1em\vrule width 0.5em height 0.69678ex depth -0.60387ex
                  \kern -0.6em \intop}\nolimits_{#1}}%
          {\mathop{\kern 0.1em\vrule width 0.5em height 0.69678ex
              depth -0.60387ex
                  \kern -0.6em \intop}\nolimits_{#1}}%
          {\mathop{\kern 0.1em\vrule width 0.5em height 0.69678ex depth -0.60387ex
                  \kern -0.6em \intop}\nolimits_{#1}}}
\def\vintslides_#1{\mathchoice%
          {\mathop{\kern 0.1em\vrule width 0.5em height 0.697ex depth -0.581ex
                  \kern -0.6em \intop}\nolimits_{\kern -0.4em#1}}%
          {\mathop{\kern 0.1em\vrule width 0.3em height 0.697ex depth -0.604ex
                  \kern -0.4em \intop}\nolimits_{#1}}%
          {\mathop{\kern 0.1em\vrule width 0.3em height 0.697ex depth -0.604ex
                  \kern -0.4em \intop}\nolimits_{#1}}%
          {\mathop{\kern 0.1em\vrule width 0.3em height 0.697ex depth -0.604ex
                  \kern -0.4em \intop}\nolimits_{#1}}}

\newcommand{\intav}{\vint}
\newcommand{\aveint}[2]{\mathchoice%
          {\mathop{\kern 0.2em\vrule width 0.6em height 0.69678ex depth -0.58065ex
                  \kern -0.8em \intop}\nolimits_{\kern -0.45em#1}^{#2}}%
          {\mathop{\kern 0.1em\vrule width 0.5em height 0.69678ex depth -0.60387ex
                  \kern -0.6em \intop}\nolimits_{#1}^{#2}}%
          {\mathop{\kern 0.1em\vrule width 0.5em height 0.69678ex depth -0.60387ex
                  \kern -0.6em \intop}\nolimits_{#1}^{#2}}%
          {\mathop{\kern 0.1em\vrule width 0.5em height 0.69678ex depth -0.60387ex
                  \kern -0.6em \intop}\nolimits_{#1}^{#2}}}

\newcommand{\re}{\mathbb{R}}
\newcommand{\rn}{\mathbb{R}^n}

\newcommand{\na}{\mathbb{N}}
\begin{document}
\title[The variation of the maximal function of a radial function
]
{The variation of the maximal function of a radial function
}
\author{Hannes Luiro}
\address{Department of Mathematics and Statistics\\
University of Jyv\"{a}skyl\"{a}\\
P.O.Box 35 (MaD)\\
40014 University of Jyv\"{a}skyl\"{a}, Finland}
\email{hannes.s.luiro@jyu.fi}
\thanks{The author was supported by the Academy of Finland, project no. 292797}
\maketitle
{\small \textbf{Abstract.} 
It is shown for the non-centered Hardy-Littlewood 
maximal operator $M$ that $\norm{DMf}_1\leq C_n\norm{Df}_1$ for all radial functions in $W^{1,1}(\rn)\,$. 
\section{Introduction}
The non-centered Hardy-Littlewood maximal operator $M$ is defined by
setting for $f\in L^1_{loc}(\rn)\,$ that
\begin{equation}\label{eq:1}
M f(x)=\sup_{B(z,r)\ni x}
\frac{1}{|B(z,r)|}\int_{B(z,r)}|f(y)|\,dy\,=:\,\sup_{B(z,r)\ni x}\intav_{B(z,r)}|f(y)|\,dy\,
\end{equation}
for every $x\in\rn\,$. The centered version of $M$, denoted by $M_c$, is defined by taking the supremum over all balls centered at $x$.
The classical theorem of Hardy, Littlewood and Wiener asserts that $M$ (and $M_c$) is 
bounded on $L^p(\rn)\,$ for $1<p\leq\infty\,$. This result is one of the cornerstones
of the harmonic analysis. While the absolute size of a maximal function is usually the principal interest, the applications 
in Sobolev-spaces and in the potential theory have motivated the active research of the regularity properties of maximal functions. 
The first observation was made by Kinnunen who verified \cite{Ki} that
$M_c$ is bounded in Sobolev-space $W^{1,p}(\rn)\,$ if $1<p\leq\infty\,$, and inequality
\begin{equation}\label{est1}
|DM_cf(x)|\leq M_c(|Df|)(x)\,
\end{equation}
holds for all $x\in\rn$.
The proof is relatively simple and inequality (\ref{est1}) (and the boundedness) holds also for $M$ and many other variants. 

The most challenging open problem in this field is so called '$W^{1,1}$-problem': Does it hold for all $f\in W^{1,1}(\rn)\,$, that
$Mf\in W^{1,1}(\rn)$ and 
\begin{equation*}
\norm{DMf}_1\leq C_n\norm{Df}_1\,?
\end{equation*}
This problem has been discussed (and studied) for example in \cite{AlPe}, \cite{CaHu}, \cite{CaMa}, \cite{HO}, \cite{HM}, \cite{Ku} and \cite{Ta}.
The fundamental obstacle is that $M$ is not bounded in $L^1$ and therefore inequality (\ref{est1}) is not enough to solve the problem.
In the case $n=1$ the answer is known to be positive, as was proved by Tanaka \cite{Ta}. For $M_c$ the problem turns out to be very complicated also when $n=1$.  
However, Kurka \cite{Ku} managed to show that the answer is positive also in this case.

The goal of this paper is to develop technology for $W^{1,1}$-problem in higher 
dimensions, where the problem is still completely open. The known proofs in the one-dimensional case are strongly based on the 
simplicity of the topology: the crucial trick 
(in the non-centered case) is that 
$Mf$ does not have a strict local maximum outside the set $\{Mf(x)=f(x)\}$. 
This fact is a strong tool when $n=1$ but is far from sufficient for higher dimensions. 


The formula for the derivative of the maximal function (see Lemma \ref{peruskama} or \cite{L}) has an important role in the paper. It says that
if $Mf(x)=\intav_{B}|f|$, $|f(x)|<Mf(x)<\infty$, and $Mf$ is differentiable at $x$, then 
\begin{equation}\label{formula}
 DMf(x)=\intav_{B}Df(y)\,dy\,.
\end{equation}
From this formula one can see immediately the validity of the estimate (\ref{est1}) for $M$. However, since $B$ is exactly the ball which gives the maximal average (for $|f|$), 
it is expected that one can derive from (\ref{formula}) much more sophisticated estimates than (\ref{est1}).
In Section \ref{sec2} (Lemma \ref{peruskama}), we perform basic analysis related 
to this issue. The key observation we make is that if $B$ is as above, then  
\begin{equation}\label{formula2}
 \int_{B}Df(y)\cdot(y-x)\,dy\,=\,0\,.
\end{equation}
In the backround of this equality stands a more general princinple, concerning other maximal operators
as well: if the value of the maximal function is attained to ball (or other permissible object) $B$, then the \textit{weighted} 
integral of $|Df|$ over $B$ is zero for a set of weights depending on the maximal operator. We believe that the utilization of this principle is a key for a possible solution of 
$W^{1,1}$-problem. 


%
%
%
%
%
%
As the main result of this paper, we employ equality (\ref{formula2}) to show that in the case of \textit{radial functions} the answer 
to $W^{1,1}$-problem is positive (Theorem \ref{main}). Even in this case the problem is evidently non-trivial and truly differs from the one-dimensional case.  
To become convinced about this, consider the important special case where $f$ is
radially decreasing ($f(x)=g(|x|)$, where $g:[0,\infty)\to \re$ is decreasing). In this case
$Mf$ is radially decreasing as well and $Mf(0)=f(0)$. 
If $n=1$, these facts immediately imply that $\norm{DMf}_1=\norm{Df}_1$, but if $n\geq 2$ this is definitely not the case: the additional estimates are necessary.  
This type of estimate for radially decreasing functions can be derived 
from (\ref{formula}) and (\ref{formula2}), saying that
\begin{equation}\label{ref1}
|DMf(x)|\leq \frac{C_n}{|x|}\intav_{B(0,|x|)}|Df(y)||y|\,dy\,.
\end{equation}
By using this inequality, the positive answer to $W^{1,1}$-problem for radially decreasing functions follows straightforwardly by 
Fubini Theorem (Corollary \ref{cor1}). 

For general radial functions, inequality (\ref{ref1}) turns out to hold only if the maximal average is achieved in a ball with radius comparable to 
$|x|$. To overcome this problem, we study 
the auxiliary maximal function $M^{I}$, defined for $f\in L^{1}_{loc}(\rn)$ by
\begin{equation*}
M^If(x)=\sup_{x\in B(z,r), r\leq|x|/4 }\intav_{B(z,r)}|f(y)|\,dy\,,
\end{equation*}
and
prove (Lemma \ref{basic3}) that for all radial $f\in W^{1,1}(\rn)$ it holds that
\begin{equation}\label{eq:eq}
 \norm{DM^If}_1\leq C_n\norm{Df}_1\,.
\end{equation}  
The proof of this auxiliary result resembles the proof of  $W^{1,1}$-problem (for $M$) in the case $n=1$. 
As the first step, we prove by straightforward calculation that for the 'endpoint operator' of $M^I$, defined by 
\begin{equation}\label{def99}
f_{/4}(x):=\sup_{x\in B(z,|x|/4)}\intav_{B(z,|x|/4)}\,|f(y)|\,dy\,,
\end{equation}
it holds that $\norm{Df_{/4}}_1\leq C\norm{Df}_1$ for all $f\in W^{1,1}(\rn)$. 
Recall again the fact that $Mf$ does not have a local maximum in $\{Mf(x)>|f(x)|\}$, leading to the estimate
$\norm{DMf}_1\leq \norm{Df}_1$ in the case $n=1$. As a multidimensional counterpart for radial functions, we show that
$M^If$ does not have a local maximum in $\{M^If(x)>\max\{|f(x)|,f_{/4}(x)\}\}$ and for every $k\in \mathbb{Z}$ it holds that
\begin{equation*}
 \int_{\{2^k\leq|y|\leq 2^{k+1}\}}DM^If(y)\,dy\,\leq\,C_n\int_{\{2^{k-1}\leq |y|\leq 2^{k+2}\}}|D|f|(y)|\,dy\,.
\end{equation*}
Estimate (\ref{eq:eq}) can be easily derived from this fact. 
The main result follows by combining (\ref{eq:eq}) and exploiting the estimate (\ref{ref1}) in $\{Mf(x)>M^{I}f(x)\}$.

\subsection*{Question}
The analysis presented in this paper raises the interest towards the study of the integrability properties of some \textit{conditional} maximal operators.
As an example, (\ref{formula}) and (\ref{formula2}) yield that $|DMf(x)|\leq \widetilde{M}(D|f|)(x)$, where 
$\widetilde{M}$ is defined for all locally integrable gradient fields $F:\rn\to\rn$ by
\begin{equation*}
\widetilde{M}F(x)= \sup\bigg\{\,\bigg|\intav_{B(z,r)}F\bigg|\,:\,x\in B(z,r)\,,\,\,\int_{B(z,r)}F(y)\cdot(y-x)dy\,=0\,\bigg\}\,.
\end{equation*}
It is clear that $\widetilde{M}F$ is bounded by $M(|F|)$, but does it hold that $\widetilde{M}$ has even better integrability properties than $M$? What about the boundedness
in the Hardy-space $H^1$ or even in $L^1$? Notice that the boundedness of $\widetilde{M}$ in $L^1$ would imply the solution to  $W^{1,1}$-problem. 
This problem is almost completely open, even in the case $n=1$. Counterexamples would be highly interesting as well.   

\textit{Acknowledgements.} The author would like to thank Antti V\"{a}h\"{a}kangas for useful comments on the manuscript and inspiring discussions.
\section{Preliminaries and general results}\label{sec2}
Let us introduce some notation. The boundary of the $n$-dimensional unit ball is denoted by $S^{n-1}$. The $s$-dimensional Hausdorff measure is denoted by 
$\mathcal{H}^s$. The volume of the $n$-dimensional unit ball is denoted by $\omega_n$ and the $\mathcal{H}^{n-1}$-measure of $S^{n-1}$ by $\sigma_n$.  The weak derivative of 
$f$ (if exists) is denoted by $Df$. 
If $v\in S^{n-1}$, then 
\begin{equation*}
D_vf(x):=\lim_{h\to 0}\frac{1}{h}(f(x+hv)-f(x))\,,
\end{equation*}
in the case the limit exists.
\begin{definition}
For $f\in L^1_{loc}(\rn)$ let
\begin{align*}
\mathcal{B}_x:=&\{B(z,r):\,x\in\bar{B}(z,r),r>0,\intav_{B}|f|=Mf(x)\}\,.
\end{align*}
\end{definition}
It is easy to see that if $f\in L^1(\rn)$ and $|f(x)|<Mf(x)<\infty$, then $\mathcal{B}_x\not=\emptyset\,$.

The following lemma is the main result of this section. We point out that below $(6)$ is especially useful in the case of radial functions.
\begin{lemma}\label{peruskama}
Suppose that $f\in W^{1,1}(\rn)$, $Mf(x)>f(x)$ and $Mf$ is differentiable at $x$. Then
\begin{enumerate}
\item
For all $v\in S^{n-1}$ and $B\in\mathcal{B}_x\,$, it holds that 
\begin{equation*}
DMf(x)=\intav_{B}D|f|(y)\,dy\,\text{ and }\,D_{v}Mf(x)=\intav_{B}D_{v}|f|(y)\,dy\,.
\end{equation*}
\item
If $x\in B$ for some $B\in \mathcal{B}_x$, then $DMf(x)=0\,.$
\item
If $x\in \partial B$, $B=B(z,r)\in \mathcal{B}_x$ and $DMf(x)\not=0$, then
\begin{equation*}
\frac{DMf(x)}{|DMf(x)|}=\frac{z-x}{|z-x|}\,.
\end{equation*}
\item If $B\in\mathcal{B}_x$, then 
\begin{equation}\label{char}
\int_{B}D|f|(y)\cdot (y-x)\,dy\,=\,0\,.
\end{equation}
\item If $x\in \partial B$, $B=B(z,r)\in \mathcal{B}_x$, then
\begin{equation*}
|DMf(x)|=\frac{1}{r}\intav_{B}D|f|(y)\cdot (z-y)\,dy\,.
\end{equation*}
\item If $B\in\mathcal{B}_x$, then
\begin{equation}\label{nice}
DMf(x)\cdot\frac{x}{|x|}=\frac{1}{|x|}\intav_{B}D|f|(y)\cdot y\,dy\,.
\end{equation}
\end{enumerate}
\end{lemma}

The proof of Lemma \ref{peruskama} is essentially based on the following auxiliary propositions. 
\begin{proposition}\label{aux1}
Suppose that $f\in W^{1,1}(\rn)$, $B$ is a ball, $h_i\in\re$ such that $h_i\to 0$ as $i\to\infty$, and $B_i=L_i(B)$, where $L_i$ are \textit{affine} mappings and 
\begin{equation*}
\lim_{i\to\infty}\frac{L_i(y)-y}{h_i}\,=\,g(y)\,.
\end{equation*}
Then 
\begin{equation}\label{fir}
\lim_{i\to \infty}\frac{1}{h_i}\bigg(\intav_{B_i}f(y)\,dy\,-\intav_{B}f(y)\,dy\,\bigg)\,=\,\intav_{B}Df(y)\cdot g(y)\,dy\,.
\end{equation}
\end{proposition}
\begin{proof}The proof is a simple calculation: 
\begin{align*}
&\frac{1}{h_i}\bigg(\intav_{B_i}f(y)\,dy\,-\intav_{B}f(y)\,dy\,\bigg)
=\frac{1}{h_i}\bigg(\intav_{L_i(B)}f(y)\,dy\,-\intav_{B}f(y)\,dy\,\bigg)\\
=&\frac{1}{h_i}\bigg(\intav_{B}f(L_i(y))-f(y)\,dy\,\bigg)
=\intav_{B}\frac{f(y+(L_i(y)-y))-f(y)}{h_i}\,dy\,\\
\approx&\intav_{B}\frac{D f(y)\cdot(L_i(y)-y)}{h_i}\,dy\,\to\intav_{B}D f(y)\cdot g(y)\,dy\,,
\end{align*}
if $i\to \infty\,$. 
\end{proof}
\begin{lemma}\label{aux2}
Let $f\in W^{1,1}(\rn)$, $x\in\rn$, $B\in\mathcal{B}_x$, $\delta>0$, and let $L_h$, $h\in[-\delta,\delta]$, be \textit{affine} mappings such that 
$x\in L_h(\bar{B})$ and
\begin{equation}\label{assump}
\lim_{h\to 0}\frac{L_h(y)-y}{h}\,=\,g(y)\,.
\end{equation}
Then
\begin{equation}\label{sec}
\int_{B} D|f|(y)\cdot g(y)\,dy\,=0\,.
\end{equation}
\end{lemma}
\begin{proof}
Let us denote $B_h:=L_h(B)$.
 By Proposition \ref{aux1} it holds that
 \begin{equation*}
 \intav_{B} D|f|(y)\cdot g(y)\,dy\,=\,\lim_{h\to 0}\frac{1}{h}\bigg(\intav_{B_h}|f|(y)-\intav_{B}|f|(y)\bigg)\,.
 \end{equation*}
Since $B\in\mathcal{B}_x$ and $x\in \bar{B}_h$, the sign of the quantity inside the large parentheses is non-positive for all $h\in[-\delta,\delta]$. 
However, the sign of $1/h$ depends on the sign of
$h$. The conclusion is that the above equality is possible only if (\ref{sec}) is valid.
\end{proof}

\subsection*{Proof of Lemma \ref{peruskama}}
\begin{enumerate}
 \item The claim is counterpart for the formula for $DM_cf$, which was first time proved in \cite{L}.
 Suppose that $B=B(z,r)\in \mathcal{B}_x$ and let $B_h:=B(z+hv,r)$. Then it holds that
 \begin{align*}
  &D_vMf(x)=\lim_{h\to 0}\frac{1}{h}(Mf(x+hv)-Mf(x))\\
  \geq& \lim_{h\to 0}\frac{1}{h}\bigg(\intav_{B_h}|f(y)|\,dy\,-\intav_{B}|f(y)|\,dy\,\bigg)\\
  =&\lim_{h\to 0}\frac{1}{h}\bigg(\intav_{B}|f(y+hv)|-|f(y)|\,dy\,\bigg)\,=\,\intav_{B_h}D_v|f|(y)\,dy\,.
 \end{align*}
On the other hand, if $B_h:= B(z-hv,r)$, then
\begin{align*}
  &D_vMf(x)=\lim_{h\to 0}\frac{1}{h}(Mf(x)-Mf(x-hv))\\
  \leq&\lim_{h\to 0}\frac{1}{h}\bigg(\intav_{B}|f(y)|\,dy\,-\intav_{B_h}|f(y)|\,dy\,\bigg)\\
  =&\lim_{h\to 0}\frac{1}{h}\bigg(\intav_{B}|f(y)|-|f(y+hv)|\,dy\,\bigg)\,=\,\intav_{B_h}D_v|f|(y)\,dy\,.
 \end{align*}
 These inequalities imply the claim.
\item If $B\in\mathcal{B}_x$ and $x\in B$, then $y\in B$ if $|y-x|$ is small enough, and thus $Mf(y)\geq Mf(x)$.
\item Let  $B=B(z,r)\in\mathcal{B}_x$, $v\in S^{n-1}$ such that $v\cdot(z-x)=0$, and  $h_i\in(0,\infty)$, $h_i\to 0$ as $i\to\infty\,$. Moreover, let us denote 
$B_i:=B(z,|z-(x+h_iv)|)$. Then it clearly holds that $x+h_iv\in \bar{B}_i$ and it is also 
easy to see that $B_i=L_i(B)$ for an affine mapping $L_i$ given by 
\begin{equation*}
 L_i(y)=y+\bigg(\frac{|z-(x+h_iv)|-|z-x|}{|z-x|}\bigg)(y-z)\,.
\end{equation*}
By the assumption $v\cdot (z-x)=0$ it follows that
\begin{equation*}
 \lim_{i\to \infty}\frac{L_i(y)-y}{h_i}\,=\,(y-z)\lim_{i\to \infty}\bigg(\frac{|z-(x+h_iv)|-|z-x|}{|z-x|}\,\bigg)\,=\,0\,.
\end{equation*}
Therefore, Proposition \ref{aux1} implies that
\begin{equation*}
\lim_{i\to \infty}\frac{1}{h_i}\bigg(\intav_{B_i}|f|(y)\,dy\,-\intav_{B}|f|(y)\,dy\,\bigg)\,=\,0\,.
\end{equation*}
This shows that $D_{v}Mf(x)=0$ for all $v$ orthogonal to $(z-x)$. In particular, it follows that $DMf(x)$ is parallel to $z-x$ or $x-z$. The final claim follows easily 
by the fact that $Mf(x+h(z-x))\geq Mf(x)$ if $0<h\leq 2$.
\item Let $B\in\mathcal{B}_x$ and $L_h(y):=y+h(y-x)\,$, $h\in\re$. Then it holds that $L_h$ is affine mapping, $L_h(x)=x$, and so $x\in L_h(B)=:B_h$, and
$(L_h(y)-y)/h=y-x\,$ for all $h\in\re\,$. Therefore, Lemma \ref{aux2} implies that
\begin{equation*}
\int_{B} D|f|(y)\cdot (y-x)\,dy\,=0\,.
\end{equation*}
\item By combining $(1)$, $(3)$ and $(4)$ the claim follows by 
\begin{align*}
|DMf(x)|&=DMf(x)\cdot\bigg(\frac{z-x}{|z-x|}\bigg)=\intav_{B} D|f|(y)\cdot\bigg(\frac{z-x}{|z-x|}\bigg)\,dy\,\\
&=\intav_{B} D|f|(y)\cdot\bigg(\frac{z-y}{|z-x|}\bigg)\,dy\,.
\end{align*}
\item The claim follows from $(1)$ and $(4)$.
\end{enumerate}
\hfill$\Box$

\section{$W^{1,1}$-problem for radial functions}
\subsection*{Radial functions and notation}
In what follows, we will interpret a radial function on $\rn$ as a function on $(0,\infty)$ in a natural way. To be more precise, if $f\in W^{1,1}_{loc}(\rn)$ is radial,
it is well known fact that  
there exists continuous function $\tilde{f}:(0,\infty)\to\re$ 
such that $\tilde{f}$ is weakly differentiable,
\begin{equation*}
\int_{0}^{\infty}|\tilde{f}'(t)|t^{n-1}\,dt\,<\infty\,,
\end{equation*}
and (by a possible redefinition of $f$ in a set of measure zero) for all $t\in(0,\infty)$  it holds that
$f(x)=\tilde{f}(t)$ and 
$D_{x/|x|}f(x)=\tilde{f}'(t)$
if $|x|=t$. In what follows, we will simplify the notation and use $f$ to denote $\tilde{f}$ as well. To avoid the possibility of misuderstanding, 
we usually use variable $t$ and notation $f'$ (instead of $Df$) when 
we are actually working with $\tilde{f}$. We also say that $f$ is radially decreasing if $f$ is radial and $f(t_1)\leq f(t_2)$ if $t_1>t_2$.
Notice also that if $f$ is radial then $Mf$ is also radial.

The following result is an easy consequence of Lemma \ref{peruskama}. 
\begin{corollary}\label{cor1}
 If $f\in W^{1,1}(\rn)$ is radially decreasing, then $DMf\in W^{1,1}(\rn)$ and $\norm{DMf}_1\leq C_n\norm{Df}_1\,$.
\end{corollary}
\begin{proof}
Since $f$ is radially decreasing, it is easy to show (the rigorous proof is left to the reader) that if $Mf(x)\not = 0$ and $B\in \mathcal{B}_x$, then 
$0\in \bar{B}$ and $\bar{B}\subset \bar{B}(0,|x|)$. Especially, we get by Lemma \ref{peruskama}, $(6)$, that
\begin{equation}\label{ref2}
|DMf(x)|\leq \frac{C_n}{|x|}\intav_{B(0,|x|)}|Df(y)||y|\,dy\,.
\end{equation}
Then the claim follows by Fubini theorem:
\begin{align*}
 &\int_{\rn}\bigg(\frac{1}{|x|}\intav_{B(0,|x|)}| Df(y)||y|\,dy\,\bigg)\,dx\\
 =&\int_{\rn}|D f(y)||y|\bigg(\int_{\rn}\frac{\chi_{B(0,|x|)}(y)}{\omega_n|x|^{n+1}}\,dx\,\bigg)\,dy\\
 =&\int_{\rn}|D f(y)||y|\bigg(\int_{\{x:|x|\geq |y|\}}\frac{1}{\omega_n|x|^{n+1}}\,dx\,\bigg)\,dy\\
 =&\int_{\rn}|D f(y)||y|\bigg(\int_{S^{n-1}}\int_{|y|}^{\infty}\frac{1}{\omega_nt^{n+1}}t^{n-1}\,dt\,d\mathcal{H}^{n-1}\bigg)\,dy\,\\
 =&\frac{\sigma_n}{\omega_n}\int_{\rn}|D f(y)||y|\bigg(\int_{|y|}^{\infty}\frac{1}{t^2}\,dt\bigg)\,dy\,\\
 =&\frac{\sigma_n}{\omega_n}\int_{\rn}|D f(y)|\,dy\,.\\
\end{align*}
\end{proof}
In the case of general radial functions, (\ref{ref1}) is in general valid (and useful) only for those
$x$ for which the radius of $B\in\mathcal{B}_x$ is comparable 
to $|x|$. As it was explained in the introduction, the main auxiliary tool in the case of general radial functions is the following result 
(recall the definition of $M^I$ in the introduction):
\begin{lemma}\label{basic3}
If $f\in W^{1,1}(\rn)$ is radial, then 
$M^{I}f\in W^{1,1}(\rn)$ and \\ $\norm{DM^{I}f}_1\leq C_n\norm{Df}_1\,$.
\end{lemma}
Before the actual proof of this result, we prove several auxiliary results.
The first of them is well known.
\begin{proposition}\label{union}
Suppose that $E\subset\re$ is open. Then there exist disjoint intervals $(a_i,b_i)$ such that 
$E=\cup_{i=1}^{\infty}(a_i,b_i)\,$ and $a_i,b_i\in\partial E\cup\{-\infty,\infty\}$ for all $i\in\N\,$.
\end{proposition}

The following auxiliary result is repeatedly utilized in the proof. The result is well known but we express the proof for readers convenience.
\begin{lemma}\label{basic}
Suppose that $\Omega\subset\rn$, $f\in W^{1,1}(\Omega)$ is continuous, $g:\Omega\to\re$ is continuous and weakly differentiable in 
$
E:=\{x\in\Omega:g(x)>f(x)\}\, 
$
, and $\int_{E}|Dg|<\infty\,$.
Then $\max\{f,g\}$ is weakly differentiable in $\Omega\,$ and 
\begin{equation*}
D(\max\{f,g\})=\chi_{E}Dg+\chi_{\Omega\cap E^c}Df\,.
\end{equation*}

\end{lemma}
\begin{proof}
Suppose that $\phi$ is a smooth test function, compactly supported in $\Omega$, $1\leq i\leq n$, $L(t)=p+te_i$, $p\in\rn\,$, and let $L$ denote the line $L(\re)$. 
By Proposition \ref{union}, $E\cap L$ can be written as a union of disjoint and open (in $\Omega\cap L$) line segments $E_j=L((a_j,b_j))$, $j\in\N$, such that 
$L(a_j),L(b_j)\in \partial E$ (with respect to $\Omega\cap L$) or $a_j=-\infty$ or $b_j=\infty$. 
In particular, $f(L(a_j))=g(L(a_j))$ if $a_j\not=-\infty$ and $f(L(b_j))=g(L(b_j))$ if $b_j\not=\infty$. Since $\phi$ is compactly supported, it follows that
\begin{align*}
&f(L(a_j))\phi(L(a_j))=g(L(a_j))\phi(L(a_j))\,\text{ and }\\
&f(L(b_j))\phi(L(b_j))=g(L(b_j))\phi(L(b_j))\,\text{ for all }j\in\N\,.
\end{align*}
Therefore, by using the assumptions for $g$, it holds that
\begin{align*}
&\int_{E_j}g(D_i\phi)d\ha^1=\int_{E_j}D_i(g\phi)d\ha^1-\int_{E_j}(D_ig)\phi d\ha^1\,\\
=&\,g(L(b_j))\phi(L(b_j))-g(L(a_j))\phi(L(a_j))-\int_{E_j}(D_ig)\phi d\ha^1\\
=&\,f(L(b_j))\phi(L(b_j))-f(L(a_j))\phi(L(a_j))-\int_{E_j}(D_ig)\phi d\ha^1\\
=&\,\int_{E_j}D_i(f\phi)d\ha^1-\int_{E_j}(D_ig)\phi d\ha^1\\
=&\,\int_{E_j}(D_if)\phi + f(D_i\phi)-(D_ig)\phi\, d\ha^1\,
\end{align*}
for all $j\in\N\,$. Then
\begin{align*}
&\int_{\Omega\cap L}\max\{f,g\}(D_i\phi)d\ha^1=\int_{E\cap L}g(D_i\phi)d\ha^1+\int_{\Omega\cap E^c\cap L}f(D_i\phi)d\ha^1\\
=&\sum_{j=1}^{\infty}\int_{E_j}g(D_i\phi)d\ha^1+\int_{\Omega\cap E^c\cap L}f(D_i\phi)d\ha^1\\
=&\int_{E\cap L}(D_if)\phi +f(D_i\phi)-(D_ig)\phi\, d\ha^1
+\int_{\Omega\cap E^c\cap L}f(D_i\phi)\,d\ha^1\\
=&\int_{\Omega\cap L}f(D_i\phi)\,d\ha^1+\int_{E\cap L}(D_if)\phi\,d\ha^1-\int_{E\cap L}(D_ig)\phi\, d\ha^1\\
=&-\int_{\Omega\cap L}(D_i f)\phi\,d\ha^1+\int_{E\cap L}(D_if)\phi\,d\ha^1-\int_{E\cap L}(D_ig)\phi\, d\ha^1\\
=&-\int_{\Omega\cap\E^c\cap L}(D_i f)\phi\,d\ha^1-\int_{E\cap L}(D_ig)\phi\, d\ha^1\,\\
=&-\int_{\Omega\cap L}\big(\chi_{E}D_ig+\chi_{\Omega\cap E^c}D_if\big)\phi\,d\ha^1\,.
\end{align*}
This implies the claim.
\end{proof}

\begin{definition}
Let $f:\Omega\to\re$, where $\Omega\subset\re$ is open. We say that $x$ is a \text{local strict maximum} of $f$ in $(a,b)\subset\Omega$, $-\infty\leq a<b\leq\infty$, if there exist $a',b'\in(a,b)$ such that 
$a'<x<b'$, $f(t)\leq f(x)$ if $t\in (a',b')$, and $\max\{f(a'),f(b')\}<f(x)$. 
\end{definition}
\begin{proposition}\label{lation}
 Suppose that $f:[a,b]\to\re$ is continuous and $c\in (a,b)$ such that $f(c)>\max\{f(a),f(b)\}$. Then $f$ has a local strict maximum on $(a,c)$. 
\end{proposition}
\begin{proof}
It is easy to see that now any maximum point $c$ ($f(c)=\max f$), which is known to exist, is also a local strict maximum of $f$.
\end{proof}
\begin{proposition}\label{deejee}
Suppose that $f:[a,b]\to\re$ is continuous and does not have a local strict maximum on $(a,b)$. Then there exists $c\in[a,b]$ such that 
 $f$ is non-increasing on $[a,c]$ and non-decreasing on $[c,b]$. 
\end{proposition}
\begin{proof}
Since $f$ is continuous, we can choose $c\in[a,b]$ such that $f(c)=\min f$. To show that $f$ is non-decreasing on $[c,b]$, let $c<y_1<y_2<b$ and assume, on the contrary, that
$f(y_2)<f(y_1)$. This implies that $f(y_1)>\max\{f(c),f(y_2)\}$, and thus $f$ has a strict local maximum on $(c,y_2)$ by Proposition \ref{lation}. This is the desired contradiction.
To show that $f$ is non-increasing on $[a,c]$, let $a<y_1<y_2<c$ and assume, on the contrary, that
$f(y_1)<f(y_2)$. This implies that $f(y_2)>\max\{f(y_1),f(c)\}$, and thus $f$ has a strict local maximum on $(y_1,c)$ by Proposition \ref{lation}. This is the desired contradiction. 
\end{proof}


Let us define for $0<a\leq b<\infty$ the annular domains 
\begin{align*}
 A_n(a,b):=&A(a,b):=\{x\in\rn\,:\,a< |x|< b\}\,\,\text{ and }\\
 A_n[a,b]:=&A[a,b]:=\{x\in\rn\,:\,a\leq |x|\leq b\}\,.
\end{align*}

\begin{lemma}\label{simppeli}
 If $f\in W^{1,1}(\rn)$ is radial, then $Mf$ does not have a local strict maximum in $\{t\in(0,\infty):\,Mf(t)>f(t)\}\,$.
\end{lemma}
\begin{proof}
Suppose, on the contrary, that $t_0\in(0,\infty)$ is a local strict maximum of $Mf$ and $Mf(t_0)>f(t_0)$. Let us choose 
\begin{align*}
 t^-:=&\sup\{t<t_0:\,Mf(t)<Mf(t_0)\}\,\text{ and }\\
 \,t^+:=&\inf\{t>t_0\,:\,Mf(t)<Mf(t_0)\}\,.
\end{align*}
By the definition of the local strict maximum, it follows that 
$t_0\in[t^-,t^+]$ and
\begin{equation}\label{eq:last}
Mf(t)=Mf(t_0) \text{ for all }t\in[t^-,t^+]\,. 
\end{equation}
Suppose that $|x|=t_0$. Since $Mf(t_0)>f(t_0)$, it follows that there exist a ball $B$ such that 
$Mf(t_0)=\intav_{B}|f|$, $x\in \bar{B}\,$. Suppose first that $B\not\subset  A[t^-,t^+]$. In this case there exists $\eps>0$ such that 
$[t^--\eps,t^-]\subset\{|y|: y\in \bar{B}\}$ or $[t^+,t^++\eps]\subset\{|y|: y\in \bar{B}\}$. Especially, it follows by the definition of $M$ that
$Mf(t)\geq \intav_{B}|f|=Mf(t_0)$ if $t\in[t^--\eps,t^-]$ or $t\in[t^+,t^++\eps]$, respectively. Obviously this contradicts with the choice of $t^-$ and $t^+$.
This verifies that $B\subset  A[t^-,t^+]$. Therefore, it holds by (\ref{eq:last}) that
\begin{equation}\label{dispose}
Mf(y)= Mf(t_0)\text{ for all }y\in B\,.
\end{equation} However,
$f(t_0)<Mf(t_0)$ also implies that there exists a ball $B'$ with positive radius such that 
$B'\subset B$ and $f<Mf(t_0)$ in $B'$. Combining this with (\ref{dispose}) yields the desired contradiction by
\begin{align*}
Mf(t_0)&=\intav_{B}|f|\leq\frac{1}{|B|}\bigg(\int_{B\setminus B'}|f|+\int_{B'}|f|\bigg)\\
&<\frac{1}{|B|}\bigg(\int_{B\setminus B'}Mf+\int_{B'}Mf(t_0)\bigg)
=Mf(t_0)\,.
\end{align*}
\end{proof}

%
%

Recall the definition of $f_{/4}$ (the endpoint operator of $M^I$, (\ref{def99})) from the introduction. Before showing the boundedness for $M^I$, we have to prove the boundedness
for $f_{/4}$.

\begin{proposition}\label{basic2}
If $f\in W^{1,1}(\rn)$, 
then $f_{/4}\in W^{1,1}(\rn)$ and $\norm{Df_{/4}}_1\leq C_n\norm{Df}_1\,$.
\end{proposition}
\begin{proof}
It is easy to check that $f_{/4}$ is Lipschitz outside the origin. Therefore, it suffices to verify the desired norm estimates for $Df_{/4}$. 
We will exploit Proposition \ref{aux1}. If $x\not=0$, we are going to show that if $h>0$ is small enough and $v\in S^{n-1}$, then 
\begin{equation}\label{aim1}
\frac{1}{h}|f_{/4}(x)-f_{/4}(x+hv)|\leq C_n\intav_{B(x,\frac{|x|}{2})}|D|f|(y)|\,dy\,.
\end{equation}
To show this, we may assume that $f_{/4}(x)>f_{/4}(x+hv)$.
Suppose that 
\begin{align*}
f_{/4}(x)=&\intav_{B(z,|x|/4)}|f(y)|dy\,\,,\,\,\,\,x\in \bar{B}(z,|x|/4)=:B\,,\\
g_h(y):=&x+hv+\frac{|x+hv|}{|x|}(y-x)\,\text{ and }\\
B_h:=&g_h(B)=B(x+hv+\frac{|x+hv|}{|x|}(z-x),|x+hv|/4)\,.
\end{align*}
Especially, $x+hv\in \bar{B}_h$. 
Moreover, it is easy to compute that
\begin{equation*}
\lim_{h\to 0}\frac{g_h(y)-y}{h}=\lim_{h\to 0}\frac{hv+\big(\,\frac{|x+hv|}{|x|}-1\,\big)\big(y-x\big)}{h}\,=\,v+\frac{v\cdot x}{|x|^2}(y-x)\,.
\end{equation*}
Then it follows by Proposition \ref{aux1} that
\begin{align*}
&\lim_{h\to 0}\,\frac{f_{/4}(x)-f_{/4}(x+hv)}{h}\leq 
\lim_{h\to 0}\frac{1}{h}\bigg(\intav_{B}|f(y)|\,dy\,-\intav_{B_h}|f(y)|\,dy\,\bigg)\,\\
=&\,\intav_{B}D|f|(y)\cdot (v+\frac{v\cdot x}{|x|^2}(y-x))\,dy\,\leq
\intav_{B}|D|f|(y)|(1+\frac{|y-x|}{|x|})\,dy\,\\
\leq &\intav_{B}(1+\frac{1}{4})|D |f|(y)|\,dy\,\leq C_n\intav_{B(x,\frac{|x|}{2})}|D |f|(y)|\,dy\,.
\end{align*}
This proves (\ref{aim1}).
Then the claim follows (e.g) by using Fubini Theorem: Let us denote below $B_x=B(x,\frac{|x|}{2})\,$.
By the above estimate, 
\begin{align*}
&\int_{\rn}|Df_{/4}(x)|\,dx\,\leq\,C_n\int_{\rn}\int_{\rn}\frac{\chi_{B_x}(y)}{|B_x|}|Df(y)|\,dx\,dy\,\\
\leq&\,C_n\int_{\rn}|Df(y)|\bigg(\int_{\{x\,:\,\frac{2|y|}{3}\leq |x|\leq 2|y|\}}|B_x|^{-1}\,dx\,\bigg)\,dy\,
\leq C'_n\norm{Df}_1\,.
\end{align*}
\end{proof}


The following estimate is well known.
\begin{proposition}
If $f\in W^{1,1}(\rn)$ is radial and $0<a<b<\infty$, then 
\begin{equation*}
\sigma_na^{n-1}\int_a^b|f'(t)|\,dt\,\leq \int_{A(a,b)}|Df(y)|\,dy\,\leq \sigma_nb^{n-1}\int_a^b|f'(t)|\,dt\,.
\end{equation*}
\end{proposition}

\subsection*{The proof of Lemma \ref{basic3}}
Let 
\begin{equation*}
 g(x)=\max\{f_{/4}(x),|f(x)|\}\,.
\end{equation*}
By Lemma \ref{basic} and Proposition \ref{basic2} it follows that $g\in W^{1,1}(\rn)$ and $\norm{Dg}_1\leq C_n\norm{Df}_1\,$.
Let
\begin{equation*}
E:=\{x\in\rn\,:\,M^If(x)>g(x)\}\,\text{ and }\,E_k:=E\cap A[2^{-k},2^{-k+1}]\,,\,\,k\in\na.
\end{equation*}
It is well known that mapping $M^If$ is locally Lipschitz in $E$ and, especially, $D(M^If)$ exists in $E$. By Lemma \ref{basic}, it suffices to show that
$\int_{E}|DM^If|\leq C_n\norm{Dg}_1\,$.

First observe that since $|f|$ is radial, 
it follows that $M^If$ and $g$ are radial as well, and continuous in $\rn\setminus\{0\}$. In particular, if 
\begin{equation*}
E_k^{\re}:=\{|x|\,:\,x\in E_k\}\,,
\end{equation*}
then $x\in E_k$ if and only if $|x|\in E^{\re}_k$. Since $E_k^{\re}$ is open, we can write 
\begin{equation*}
 E_k^{\re}=\cup_{i=1}^{\infty}(a_i,b_i)\,,
\end{equation*}
such that $a_i<b_i$, $(a_i,b_i)$ are pairwise disjoint and $a_i,b_i\in\partial E_k^{\re}$. In the other words,
\begin{equation*}
E_k=\bigcup_{i=1}^{\infty}A(a_i,b_i)\,,
\end{equation*}
and (by the definition of $E_k$) for all $i\in\na$ it holds that
\begin{equation}\label{cases}
 M^If(x)=g(x)\text{ if }|x|=a_i>2^{-k}\text{ and }M^If(x)=g(x)\text{ if }|x|=b_i<2^{-k+1}\,.
\end{equation}
Moreover, since $M^{I}f>f$ in $E_k$, Lemma \ref{simppeli} says that $M^{I}f$ does not have a strict local maximum in $E_k^{\re}$. In particular, by Proposition 
\ref{deejee} 
there exist $c_i\in(a_i,b_i)$ such that
\begin{align*}
 \int_{A(a_i,b_i)}DM^{I}f(y)\,dy\,&\leq \,\sigma_n b_i^{n-1}\int_{a_i}^{b_i}|(M^{I}f)'(t)|\,dt\,\\
 &= \,\sigma_n b_i^{n-1}(M^If(a_i)-M^If(c_i)+M^If(b_i)-M^If(c_i))\\
&\leq\,\sigma_n  b_i^{n-1}(M^If(a_i)-g(c_i)+M^If(b_i)-g(c_i))\,.
\end{align*}
Combining this with (\ref{cases}) implies that if $2^{-k}<a_i<b_i<2^{-k+1}$, then
 \begin{align*}
\int_{A(a_i,b_i)}DM^{I}f(y)\,dy\,&\leq
 \sigma_n b_i^{n-1}(g(a_i)-g(c_i)+g(b_i)-g(c_i))\\
&\leq \sigma_nb_i^{n-1}\int_{a_i}^{b_i}|g'(t)|\,dt\,\leq\,\bigg(\frac{b_i}{a_i}\bigg)^{n-1}\int_{A(a_i,b_i)}|Dg(y)|\,dy\,\\
&\leq\,2^{n-1}\int_{A(a_i,b_i)}|Dg(y)|\,dy\,.
 \end{align*}
For the case $a_i=2^{-k}$ or $b_i=2^{-k+1}$, we employ the fact
\begin{equation*}
M^If(2^{-k}),M^If(2^{-k+1})\leq \sup_{y\in A(2^{-k-1},2^{-k+2})}g(y)\,
\end{equation*}
to obtain the estimates ($a_i=2^{-k}$ or $b_i=2^{-k+1}$)
\begin{align*}
\int_{A(a_i,b_i)}DM^{I}f(y)\,dy\,&\leq
 \sigma_n b_i^{n-1}(M^If(a_i)-g(c_i)+M^If(b_i)-g(c_i))\\
&\leq \sigma_nb_i^{n-1}\int_{2^{-k-1}}^{2^{-k+2}}|g'(t)|\,dt\,\\
&\leq\,2^{3(n-1)}\int_{A(2^{-k-1},2^{-k+2})}|Dg(y)|\,dy\,.
 \end{align*}
Combining these estimates implies that
\begin{align*}
&\int_{E_k}|DM^{I}f(y)|\,dy\,=\sum_{i=1}^{\infty}\int_{A(a_i,b_i)}|DM^{I}f(y)|\,dy\,\\ 
\leq &\,2^{n-1}\sum_{i=1}^{\infty}\bigg[\int_{A(a_i,b_i)}|Dg(y)|\,dy\,\bigg]\,+2(2^{3(n-1)})\int_{A(2^{-k-1},2^{-k+2})}|Dg(y)|\,dy\,\\
\leq &\,2^{3n}\int_{A(2^{-k-1},2^{-k+2})}|Dg(y)|\,dy\,.
\end{align*}
Therefore, 
\begin{align*}
\int_{E}|DM^If(y)|\,dy\,&\leq \sum_{k\in\Z}\int_{E_k}|DM^{I}f(y)|\,dy\\
&\leq 2^{3n}\sum_{k\in\Z}\int_{A(2^{-k-1},2^{-k+2})}|Dg(y)|\,dy\\
&= \,3(2^{3n})\sum_{k\in\Z}\int_{A(2^{-k},2^{-k+1})}|Dg(y)|\,dy\,=\,3(2^{3n})\norm{Dg}_1\,.
\end{align*}
This completes the proof.
\hfill$\Box$

Then we are ready to prove our main theorem.

\begin{theorem}\label{main}
If $f\in W^{1,1}(\rn)$ is radial, then $Mf\in W^{1,1}(\rn)$ and 
$\norm{DMf}_1\leq C_n\norm{Df}_1\,$.
\end{theorem}
\begin{proof}
Let 
\begin{equation*}
E:=\{x\in\rn\,:\,Mf(x)>M^{I}f(x)\,,\,\,\,DMf(x)\not= 0\,\}.
\end{equation*}
It is well known that $Mf$ is locally Lipschitz in $\{Mf(x)>f(x)\}$, implying the existence of $DMf$ in $\{Mf(x)>f(x)\}$. 
Since $Mf\geq M^{I}f(x)$, it holds that $Mf(x)=\max\{Mf(x),M^{I}f(x)\}$. 
Therefore, the theorem follows by Lemmas \ref{basic} and \ref{basic3}, if we can show that 
\begin{equation}\label{goal}
\int_{E}|DMf(y)|\,dy\,\leq C_n\norm{Df}_1\,.
\end{equation}
To show this, 
observe first that
for all $x\in E$ there exist $r_x>\frac{|x|}{4}$ and $z_x\in\rn$ such that $x\in B(z_x,r_x)\in \mathcal{B}_x$. Moreover, since $DMf(x)\not=0$, Lemma \ref{peruskama} 
($(2)$ and $(3)$) says that
$x\in\partial B(z_x,r_x)$ and $DMf(x)/|DMf(x)|=(z_x-x)/|z_x-x|$. On the other hand, $Mf$ is radial and so $DMf(x)/|DMf(x)|=\pm x/|x|$. We conclude that 
\begin{equation*}
B_x=B(c_xx,|c_xx-x|)\text{ for some }c_x\in\re\,.
\end{equation*}
Observe that $r_x=|c_xx-x|=|c_x-1||x|> |x|/4$ by the assumption, and thus $|c_x-1|>1/4\,$. Moreover, it holds that $c_x\geq -1$. To see this, observe that
if $c_x<-1$, then $-x\in B_x$ and, since $Mf$ is radial, $B_x\in\mathcal{B}_{-x}$, implying by Lemma \ref{peruskama} that $0=DMf(-x)=DMf(x)$, which contradicts with 
the assumption $x\in E$. Summing up, we can write $E=E_+\cup E_-$, where 
\begin{equation*}
E_+=\{x\in E\,:\,c_x> 1+1/4\,\}\,\text{ and }\,E_-=\{x\in E\,:\,-1\leq c_x<3/4\,\}\,.
\end{equation*}
We are going to use different estimates for $DMf(x)$ in $E_+$ and $E_-\,$.  
Since $|DMf(x)|=|DMf(x)\cdot \frac{x}{|x|}|$, it follows from Lemma \ref{peruskama} (\ref{nice}) that
\begin{equation*}
|DMf(x)|\leq \frac{1}{|x|}\intav_{B_x}|D|f|(y)||y|\,dy\,.
\end{equation*}
This estimate will be used in $E_-$, while 
in $E_+$ we will use (easier) estimate $|DMf(x)|\leq \intav_{B_x}|D|f||\,$ (Lemma \ref{peruskama}, $(1)$). We get that
\begin{align*}
 &\int_{E}|DMf(x)|\,dx\,\leq\int_E\chi_{E_+}(x)|DMf(x)|+\chi_{E_-}(x)|DMf(x)|\,dx\,\\
 \leq & \int_E\chi_{E_+}(x)\bigg(\intav_{B_x}|D|f|(y)|dy\,\bigg)+\chi_{E_-}(x)\bigg(\intav_{B_x}|D|f|(y)|\frac{|y|}{|x|}\,dy\,\bigg)\,dx\,\\
 =& \int_E\int_{\rn}\frac{\chi_{E_+}(x)\chi_{B_x}(y)|D|f|(y)|}{|B_x|}+\frac{\chi_{E_-}(x)\chi_{B_x}(y)|D|f|(y)||y|}{|B_x||x|}\,dy\,dx\,\\
 =&\int_{\rn}|D|f|(y)|\bigg(\int_{E_+}\frac{\chi_{B_x}(y)}{|B_x|}\,dx\,+\int_{E_-}\frac{\chi_{B_x}(y)|y|}{|B_x||x|}\,dx\,\bigg)\,dy.
\end{align*}
If $y\in B_x$ and $x\in E_+$, it follows from the definition of $E_+$ that $|x|\leq |y|$. Moreover, 
$y\in B_x$ and $x\in E$ imply also that $r_x\geq \max\{|y-x|,\frac{|x|}{4}\}\geq \frac{|y|}{6}\,$. This implies the estimate
\begin{equation*}
\int_{E_+}\frac{\chi_{B_x}(y)}{|B_x|}\,dx\,\leq \int_{B(0,|y|)}\frac{dx}{\omega_n(|y|/6)^n}\leq C_n\,, \text{ for all }y\in\rn\,.
\end{equation*}
On the other hand, if $x\in E_-$, then $\,-1\leq c_x<3/4$ especially implies that $B_x\subset B(0,|x|)$. Therefore, if $x\in E_-$ and $y\in B_x$, then $y\in B(0,|x|)$, and thus 
$|x|\geq |y|\,$. Recall also that $r_x\geq \frac{|x|}{4}\,$. Combining these yields that 
\begin{align*}
 \int_{E_-}\frac{\chi_{B_x}(y)|y|}{|B_x||x|}\,dx\,\leq |y|\int_{\rn\setminus B(0,|y|)}\frac{dx}{\omega_n(|x|/4)^{n+1}}=C'_n|y|\int_{|y|}^{\infty}\frac{dt}{t^2}=C'_n\,,
\end{align*}
for all $y\in\rn\,$.
This completes the proof.
\end{proof}

\end{document}